\documentclass[a4paper]{llncs}
\usepackage{llncsdoc}
\usepackage[T1]{fontenc}
\usepackage[utf8]{inputenc}
\usepackage{graphicx}
\usepackage{hyperref}
\usepackage{amssymb}
\usepackage{amsmath,bm}
\usepackage{lmodern} 
\usepackage{textcomp} 
\usepackage{sectsty}
\usepackage{fancybox}
\usepackage{listings}
\usepackage{charter}
\usepackage[titletoc]{appendix}
\usepackage{epsfig}
\usepackage{tabularx}
\usepackage{authblk}
\usepackage{enumitem}
\usepackage[ruled,vlined]{algorithm2e}
\usepackage{latexsym}
\usepackage{alltt}
\usepackage{setspace}
\usepackage{datetime}
\newdateformat{mydate}{\monthname[\THEMONTH], \THEYEAR}
\def\qed{\hfill$\Box$}

\newtheorem{observation}{\textbf{Observation}}

\usepackage[showframe=false, left=3.75cm, top=3cm, right=3cm, bottom=3cm ]{geometry}
\makeatletter

\def\@seccntformat#1{\@ifundefined{#1@cntformat}%
   {\csname the#1\endcsname\quad}  
   {\csname #1@cntformat\endcsname}
}
\let\oldappendix\appendix 
\renewcommand\appendix{%
    \oldappendix
    \newcommand{\section@cntformat}{\appendixname~\thesection\quad}
}
\makeatother

\begin{document}

\renewcommand\Authands{ and }
\title{Induced-bisecting families of bicolorings for hypergraphs}
%

\author{Niranjan Balachandran$^1$\and
Rogers Mathew$^2$\and
Tapas Kumar Mishra$^2$ \and \\
Sudebkumar Prasant Pal$^2$}

\institute {Department of Mathematics, Indian Institute of Technology, Bombay 400076, India \email{niranj@iitb.ac.in}\and
Department of Computer Science and Engineering, Indian Institute of Technology, Kharagpur 721302, India \\
\email{\{rogers,tkmishra,spp\}@cse.iitkgp.ernet.in}}

\maketitle
\begin{abstract}
%
Two $n$-dimensional vectors $A$ and $B$, $A,B \in \mathbb{R}^n$,
 are said to be \emph{trivially orthogonal} if in every coordinate $i \in [n]$, at least one of $A(i)$ or $B(i)$ is zero. 
 Given the $n$-dimensional Hamming cube $\{0,1\}^n$, we study the minimum cardinality of a set $\mathcal{V}$ of $n$-dimensional  $\{-1,0,1\}$ vectors, each containing exactly $d$ non-zero entries, such that every 
 `possible'  point $A \in \{0,1\}^n$ in the Hamming cube has some $V \in \mathcal{V}$ which is orthogonal, but not trivially orthogonal, to $A$. We give asymptotically tight lower and (constructive) upper bounds for such a set $\mathcal{V}$ except for the even values of $d \in \Omega(n^{0.5+\epsilon})$, for any $\epsilon$, $0< \epsilon \leq 0.5$.
\end{abstract}
Keywords: Hamming cube, Covering, Hyperplane, Bicoloring



\section{Introduction}
\label{sec:intro}

Two $n$-dimensional vectors $A$ and $B$, $A,B \in \mathbb{R}^n$, are said to be \emph{trivially orthogonal} if in every coordinate $i \in [n]$, at least one of $A(i)$ or $B(i)$ is zero. 
The vectors $A$ and $B$ are \emph{non-trivially orthogonal} if they are orthogonal, but not trivially orthogonal.
Consider the following problem:
"Given the $n$-dimensional Hamming cube $\{0,1\}^n$, what is the minimum cardinality of a subset $\mathcal{V}$ of  $n$-dimensional  $\{-1,0,1\}$ vectors, each containing exactly $d$ non-zero entries,  such that every point $A \in \{0,1\}^n$ in the Hamming cube has some $V \in \mathcal{V}$ which is  non-trivially orthogonal to $A$?".
It is not hard to see that the all-zero  vector and the unit vectors $\{(1,0,\ldots,0), (0,1,\ldots,0), \ldots, (0,0,\ldots,1)\}$
can never have any non-trivially orthogonal vector in $\{-1,0,1\}^n$.
Additionally, the all-ones vector $(1,\ldots,1)$ cannot be non-trivially orthogonal to any vector in $\{-1,0,1\}^n$ consisting of exactly $d$ non-zero entries, when $d$ is odd.
We call the vectors $(0,\ldots,0), (1,0,\ldots,0), \ldots, (0,0,\ldots,1)$ (and additionally, $(1,\ldots,1)$ when $d$ is odd) as \textit{trivial}.
Since no $n$-dimensional $\{-1,0,1\}$ vector with exactly one non-zero entry
is non-trivially orthogonal to any non-trivial point of the Hamming cube,
we assume that $d \geq 2$ in the rest of the paper.

\begin{definition}\label{def:1}
Let $2 \leq d \leq n$, where $d$ and $n$ are integers. We define
$\beta^d(n)$ as the minimum cardinality of a subset $\mathcal{V}$ of $n$-dimensional  $\{-1,0,1\}$ vectors, each containing exactly $d$ non-zero entries, such that every non-trivial point in the Hamming cube $\{0,1\}^n$ has a non-trivially orthogonal vector $V \in \mathcal{V}$.
\end{definition}
In this paper, we study the problem of estimation of bounds for $\beta^d(n)$.

We now define a general version of the 
aforementioned problem in terms of bicolorings of a hypergraph.
Let $G$ be a hypergraph on the vertex set $[n]$.
Corresponding to the trivial vectors/points of the Hamming cube, the singleton sets and the empty set (and additionally, the set $[n]$ when $d$ is odd) are the \emph{trivial hyperedges} or \emph{trivial subsets} of $[n]$.
Let $X^S$ denote a $\pm 1$ bicoloring of vertices of $S \subseteq [n]$,
i.e.  $X^S : S \rightarrow \{+1,-1\}$, for some $S \subseteq [n]$.
We abuse the notation to denote the subset of vertices colored with +1 (-1) with 
respect to bicoloring $X^S$ as $X^S(+1)$ (resp., $X^S(-1)$).
\begin{definition}\label{def:2}
Given a hypergraph $G$,
a hyperedge $A \in E(G)$ is said to be \emph{induced-bisected} by a bicoloring $X^S$ of a subset $S \subseteq V(G)$,
if $|A \cap X^S(+1)|=|A \cap X^S(-1)| \neq 0$.
A set $ \mathcal{X}=\{X^{S_1},\ldots, X^{S_t}\}$ of $t$ bicolorings is called an \emph{induced-bisecting family of order $d$} for $G$ if
\begin{enumerate}
	\item each $S_i \subseteq [n]$ has exactly $d$ vertices, $1 \leq i \leq t$, $2 \leq d \leq n$, and
	\item every non-trivial hyperedge $A \in E(G)$ is induced-bisected by at least one $X^{S_i}$, $1 \leq i \leq t$.
\end{enumerate}
Let $\beta^d(G)$ denote the minimum cardinality of an induced-bisecting family of order $d$ for hypergraph $G$.
\end{definition}

From Definitions \ref{def:1} and \ref{def:2}, it is clear that
the maximum of $\beta^d(G)$ over all hypergraphs $G$ on $[n]$ is $\beta^d(n)$.

\begin{example}
	 Let $\mathcal{H}$ be the hypergraph with all the $2^n-n-1$ non-trivial subsets of $[n]$ as hyperedges and let $d=2$.
	 For any $S \in \binom{[n]}{2}$, let $X^S$ color one point in $S$ with color +1 and the other with -1.
	 Observe that $\mathcal{X}=\{X^S|S \in \binom{[n]}{2}\}$ forms an induced-bisecting family of order 2 for $\mathcal{H}$.
	 $\beta^2(\mathcal{H}) \leq \binom{n}{2}$. 
	 Moreover, this upper bound is also tight: if $X^{\{a,b\}} \not \in \mathcal{X}$, $\{a,b\} \in \binom{[n]}{2}$, then
	 the hyperedge $\{a,b\} \in\mathcal{H} $ cannot be induced-bisected.	
\end{example}

\subsection{Relations to existing work}

The problem addressed in this paper can be viewed as a generalization of the problem of \emph{bisecting families}\cite{Niranj2016}.
Let $n \in \mathbb{N}$ and let $\mathcal{A}$ be a family of subsets of $[n]$.
Another family $\mathcal{B}$ of subsets of $[n]$ is called a \emph{bisecting family} for $\mathcal{A}$,
if for each  $A \in \mathcal{A}$, there exists a $B \in \mathcal{B}$
such that $|A \cap B| \in \{\lceil \frac{|A|}{2}\rceil,\lfloor \frac{|A|}{2}\rfloor\}$.
In the bicoloring terminology,
letting $S=[n]$, $X^S(+1)=B$, $X^S(-1)=[n]\setminus B$,
the bisecting family $\mathcal{B}$ maps to a collection $\mathcal{X}$ of bicolorings such that 
for each $A \in \mathcal{A}$, there exists a bicoloring $X \in \mathcal{X}$
such that $|A \cap X(+1)|-|A \cap X(-1)| \in \{-1,0,1\}$.
In \cite{Niranj2016}, the authors define
$\beta_{[\pm 1]}(n)$ as the minimum cardinality of a bisecting family for
the family consisting of all the non-empty subsets of $[n]$, and they 
prove that $\beta_{[\pm 1]}(n)=\lceil \frac{n}{2}\rceil$ \cite{Niranj2016}.
Note that when $d=n$ and $\mathcal{A}_e$ denotes the family of non-trivial even subsets of $[n]$, any induced bisecting family of order $d$ for $\mathcal{A}_e$ is a bisecting family for the family consisting of all the non-empty subsets of $[n]$.
In other words,
$\beta^n(\mathcal{A}_e)=\beta_{[\pm 1]}(n)$.
However, observe that when $d=n$, i.e. $S=[n]$, no odd subset of $[n]$ can be induced-bisected: 
this follows from the fact that for any odd subset $A$,
$|A \cap X^S(+1)|-|A \cap X^S(-1)|$ is odd.

An affine hyperplane is a set of vectors $H(a,b)=\{x \in \mathbb{R}^n: \left< a,x\right>=b\}$,
where $a \in \mathbb{R}^n$, $b \in \mathbb{R}$. Covering the $\{0,1\}^n$ Hamming cube
with the minimum number of affine hyperplanes has been well studied - a point $x \in \{0,1\}^n$ is said to be \emph{covered} by a hyperplane
$H(a,b)$ if $\left< a,x\right>=b$.
Without any further restriction, note that $H(e_1,0)$ and $H(e_1,1)$
covers every point on the $\{0,1\}^n$ Hamming cube, where $e_1=(1,0,\ldots,0)$ is the first unit vector.
Alon and F\"{u}redi \cite{Alon1993} show that the covering-by-hyperplanes problem becomes substantially nontrivial 
under the restriction that only the nonzero vectors are covered.
They demonstrated, using the notion of 
Combinatorial Nullstellensatz \cite{AlonNull1999}, that
we need at least $n$ affine hyperplanes when the zero vector remains uncovered.
 This can be achieved by the set of hyperplanes $\{H(e_i,1)\}$, where $e_i$ is the $i$th unit vector, $1 \leq i \leq n$. 
  Many other extensions of this covering problem involving other restrictions have been studied in detail (see 
  \cite{Linial2005,sax2013,saks1993}).
The problem of bisecting families \cite{Niranj2016}
imposes the following constraints on the minimum cardinality set of covering hyperplanes $\{H_i(a_i,b_i)\}$:
(i) $b_i \in \{-1,0,1\}$;
(ii) $a_i \in \{-1,1\}^n$.
The problem of induced-bisecting families 
puts stronger restrictions not just on the hyperplanes, but also on the definition of `covering' by
a hyperplane $\{H_i(a_i,b_i)\}$:
(i) $b_i=0$;
(ii) $a_i$ consists of exactly $d$ non-zero coordinates, $a_i \in \{-1,0,1\}^n$ and $d \in [n]$;
(iii) we say a point $x$ is \emph{covered} by a hyperplane $H(a,b)$ when $a$ is nontrivially orthogonal to $x$.

\subsection*{Main result}

In this paper, we establish the following theorem.

\begin{theorem}\label{thm:main}
	Let $2 \leq d \leq n$, where $d$ and $n$ are integers. Then,
	$\frac{2n(n-1)}{d^2} \leq \beta^d(n) \leq \binom{\lceil\frac{2(n-1)}{d-1} \rceil}{2}+\lceil\frac{n-1}{d-1}\rceil (d+1)$.
	Moreover, $\beta^d(n) \geq n-1$, when $d$ is odd.
\end{theorem}
This establishes asymptotically tight bounds on $\beta^d(n)$ for all values of $n$, when $d$ is odd.
Moreover, the bound is asymptotically tight when $d \in O(\sqrt{n})$, even if $d$ is even.
However, when $d \in \Omega(n^{0.5+\epsilon})$ and $d$ is even, the above lower bound may not be asymptotically tight,
for any  $\epsilon$, $0< \epsilon \leq 0.5$.

\section{Lower Bounds}

Let $\mathcal{H}$ denote the hypergraph consisting of 
all the non-trivial subsets of $[n]$. 
Let the set $\mathcal{X}=\{X^{S_1},\ldots, X^{S_t}\}$ of bicolorings be any optimal induced-bisecting family of order $d$
for $\mathcal{H}$, where $t \in \mathbb{N}$.

Considering only the two sized subsets of $[n]$,
we note that every two element hyperedge $\{a,b\}$, $a$,$b \in [n]$,
must lie in some $S_i$, $S_i \in \{S_1,\ldots,S_t\}$; 
otherwise, no bicoloring in $\mathcal{X}$ can induced-bisect $\{a,b\}$. 
So, it follows that $\sum_{X^S \in \mathcal{X}} \binom{d}{2} \geq \binom{n}{2}$, i.e.,
$\beta^d(n) \geq \frac{n(n-1)}{d(d-1)}.$
A constant factor improvement in the lower bound 
can be obtained by the following observation: 
the maximum number of two element subsets $\{a,b\}$ that can be induced-bisected by
any $X^S \in \mathcal{X}$, $|S| =d$, is $\frac{d^2}{4}$. 
So, we have the following proposition.
\begin{proposition}\label{prop:low}
	$\beta^d(n) \geq \frac{2n(n-1)}{d^2}.$
\end{proposition}

Observe that when $d$ is large, say $d \in \Omega(n^{0.5+\epsilon})$, where $0 < \epsilon \leq 0.5$,
Proposition \ref{prop:low} only yields a sublinear lower bound.
When $d$ is odd, we can prove a general lower bound of $n-1$ on $\beta^d(n)$ using the following version of Cayley-Bacharach  theorem by Riehl and Graham \cite{riehl2003} on the maximum number of common zeros between $n$ quadratics and
any polynomial $P$ of smaller degree.

\begin{theorem}\cite{riehl2003}
	\label{thm:cayley}
	Given the $n$ quadratics in $n$ variables $x_1(x_1-1),\ldots,x_n(x_n-1)$ with $2^n$
	common zeros, the maximum number of those common zeros a polynomial $P$ of
	degree $k$ can go through without going through them all is $2^n-2^{n-k}$.
\end{theorem}

\begin{lemma}\label{lemma:lin}
	$\beta^d(n) \geq n-1$, when $d$ is odd.
\end{lemma}

\begin{proof}
	
Let $\mathcal{B}$ be a minimum-cardinality induced-bisecting family for all the non-trivial subsets
 $A \subseteq [n]$.
 Let $R_B$ denote the $n$-dimensional vector representing the bicoloring $B \in \mathcal{B}$, i.e. $R_B \in \{-1,0,1\}^n$ and $R_B$ contains exactly $d$ nonzero entries.
 Consider the polynomials $M(X)$, $N(X)$, and $P(X)$, $X \in \{0,1\}^n$.

\begin{align}
M(X=(x_1,\ldots,x_n))= \prod_{B \in \mathcal{B}}  <R_B,X>.\\
N(X=(x_1,\ldots,x_n))= \sum_{i=1}^n x_i -1. \\
P(X)= M(X)N(X).
\end{align}
Let $X_A$ denote the 0-1 $n$-dimensional incidence vector corresponding to $A \subseteq [n]$.
Note that $M(X_A)$ vanishes for each $A \subseteq [n]$  except  
(i) the all 1's vector, $(1,\ldots,1)$, since $d$ is odd, and
(ii) possibly the singleton sets.
Since $N(X_A)$ vanishes for all singleton sets, 
$P(X_A)$ vanishes on all subsets  $A \subseteq [n]$ except for the set $[n]$ (corresponding to the the all 1's vector).
Since the degree of $P$ is $|\mathcal{B}|+1$ and $P$ in non-zero only at $X_A=(1,\ldots,1)$, 
using Theorem \ref{thm:cayley}, we have $|\mathcal{B}| \geq n-1$.

\qed
\end{proof}

However, when $d$ is even, the above lower bounding technique does not work since
the polynomial $M$ may vanish at every point of the Hamming cube $\{0,1\}^n$. In this case, we can obtain a lower bound of $\Omega(\sqrt{d})$ by considering the maximum number of hyperedges that can be induced-bisected by a single bicoloring.

\section{Induced-bisecting families when $n$ is $d+1$}
\label{sec:n}

In what follows, we consider the hypergraph $\mathcal{H}$ consisting of all the 
non-trivial hyperedges of $[n]$, where $n = d+1$ and demonstrate a construction 
of an induced-bisecting family of order $d$ of cardinality $d+1$.

\begin{theorem}\label{thm:n-1}
	Let $d$ be an integer greater than 1. Then, 
	$d \leq \beta^d(d+1) \leq d+1$. Moreover,
	$\beta^{d}(d+1) = d+1$, when $d$ is even.
\end{theorem}

\begin{proof}
	
	We consider the cases when $d$ is even and $d$ is odd separately.
	We start our analysis with the case	when $d$ is even.
	Let $v_1,\ldots,v_{d+1}$ denote the $d+1$ vertices.
	Consider a circular clockwise arrangement of $d+1$ slots, namely $P_1,\ldots,P_{d+1}$ in that order.
	The slots $P_1$ to $P_{\frac{d}{2}}$ are colored with +1, slots $P_{\frac{d}{2}+2}$ to $P_{d+1}$ are colored with -1, and only slot $P_{\frac{d}{2}+1}$ remains uncolored.
	Each slot can contain exactly one vertex and each vertex takes the color of the slot it resides in.
	As for the initial configuration, let $v_i \in P_i$, for $1 \leq i \leq d+1$. This configuration gives the coloring $X_1$, where
	(i) $X_1(+1)=\{v_1,\ldots, v_{\frac{d}{2}}\}$, (ii) $X_1(-1)=\{v_{\frac{d}{2}+2},\ldots, v_{d+1}\}$, and, (iii) the vertex $v_{\frac{d}{2}+1}$ remains uncolored.
	We obtain the second coloring $X_2$ from $X_1$ by one clockwise rotation of the 
	vertices in the circular arrangement.
	Therefore, we have, $X_2(+1)=\{v_{d+1},v_1,\ldots, v_{\frac{d}{2}-1}\}$, $X_2(-1)=\{v_{\frac{d}{2}+1},\ldots, v_{d}\}$; the vertex $v_{\frac{d}{2}}$ remains uncolored.
	See Figure \ref{fig:upp1} for an illustration.
	Similarly, repeating the process $d$ times, we obtain the set $\mathcal{X}=\{X_1,\ldots,X_{d+1}\}$  of bicolorings.
	We have the following observations.
	
	\begin{figure}
		\centering
		\includegraphics[scale=0.6]{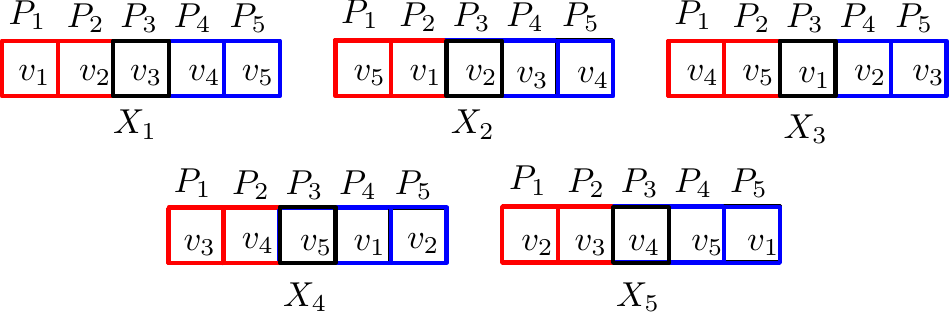}
		\caption{Vertices in (i) $P_1$ and $P_2$ are colored with +1,
			(ii) $P_4$ and $P_5$ are colored with -1; the vertex in $P_3$ remains uncolored.
			$\mathcal{X}=\{X_1,\ldots,X_5\}$ is an induced bisecting family when $n=d+1=5$.}
		\label{fig:upp1}
	\end{figure}

	\begin{observation}\label{obs:lt1}
		If $\mathcal{X}$ induced-bisects every non-trivial odd subset
		of $[d+1]$, then $\mathcal{X}$ induced-bisects every non-trivial even subset	of $[d+1]$ as well.
	\end{observation}
	
	To prove the observation, consider an	
	even hyperedge $A_e \subset [d+1]$,
	and let $X \in \mathcal{X}$ be the bicoloring that induced-bisects the odd hyperedge $\bar{A_e}=[d+1] \setminus A_e$.
	Note that one vertex in $\bar{A_e}$ remains uncolored under $X$.
	Otherwise, $\bar{A_e}$ cannot get induced-bisected under $X$.
	Since $|X(+1)|=\frac{d}{2}$ and $|\bar{A_e}\cap X(+1)|=\frac{|\bar{A_e}|-1}{2}$,
	it follows that $|A_e \cap X(+1)|=|X(+1)\setminus (\bar{A_e}\cap X(+1))|=\frac{d}{2}-\frac{|\bar{A_e}|-1}{2}$.
	Similarly, $|A_e \cap X(-1)|=\frac{d}{2} -\frac{|\bar{A_e}|-1}{2}$.
	So, $A_e$ is induced-bisected under $X$.
	This completes the proof of Observation \ref{obs:lt1}.
	
	Therefore, it suffices to prove that $\mathcal{X}$ induced-bisects
	every non-trivial odd subset of $[d+1]$.
	For the sake of contradiction, assume that $A$ is an odd hyperedge not induced-bisected by $\mathcal{X}$.
	Let $c_i=|A \cap X_{i+1}(+1)|-|A \cap X_{i+1}(-1)|$, for $0 \leq i \leq d$. 
	All additions/subtractions in the subscript of $c$ are modulo $d+1$.
	Our assumption implies that $c_i \neq 0$ for all $0 \leq i \leq d$.
	
	\begin{observation}\label{obs:lt2}
		$|c_i - c_{i+1}| \leq 2$,  for $0 \leq i \leq d$. 
		Furthermore, if $c_i > c_{i+1}$ and $c_i$ is odd, then $c_i - c_{i+1}=1$.
	\end{observation}

	The first part of Observation \ref{obs:lt2}
	follows from the construction and we omit the details for brevity.
	Note that when $c_i$ is odd, the element in $P_{\frac{d}{2}+1}$ cannot belong to the odd hyperedge $A$.
	This takes care of the second part of Observation \ref{obs:lt2}.
	
	Observe that 
	a bicoloring $X_j \in \mathcal{X}$, $1 \leq j \leq d+1$, induced-bisects the odd hyperedge $A$ if and only if $c_j$ is 0.
	We know that bicoloring $X_2$ ($X_{i+1}$) is obtained from $X_1$ ($X_i$, respectively) by one clockwise rotation of vertices
	in the circular arrangement. Thus, during the construction of bicolorings $X_1$ through $X_{d+1}$, we perform a full rotation of
	the vertices with respect to their starting arrangement in $X_1$.
	So, it follows that 
	there exist $i$ and $j$ such that $c_i$ is positive and $c_{i+j}$ is negative.
	Combined with the second part of Observation \ref{obs:lt2},
	this implies the existence of 
	an index $p$ such that $c_p=0$.
	This is a contradiction to the assumption that $A$ is not induced-bisected by $\mathcal{X}$.
	Therefore, every odd subset of $[d+1]$ is induced-bisected 
	by $\mathcal{X}$, and using Observation \ref{obs:lt1},
	the upper bound on $\beta^{d}(d+1)$ follows.
	
	To see that the upper bound is tight, observe that exactly one 
	$d$-sized hyperedge can get induced bisected under a single bicoloring - the hyperedge missing the uncolored vertex.
	This completes the proof of Theorem \ref{thm:n-1} for even values of $d$.
	%
	
	Recall that along with the empty set and the singleton sets, the set $[d+1]$ becomes trivial when $d$ is odd.
	When $d$ is odd, 
	the slots $P_1$ to $P_{\frac{d+1}{2}-1}$ are colored with +1, slots $P_{\frac{d+1}{2}+1}$ to $P_{d+1}$ are colored with -1, and only slot $P_{\frac{d+1}{2}}$ remains uncolored.
	If we generate the bicolorings $\{X_1,\ldots,X_{d+1}\}$ as in the proof for the even values of $d$,
	by similar arguments, it can be shown that
	$\{X_1,\ldots,X_{d+1}\}$ is indeed an induced-bisecting family for
	the hypergraph consisting of all the non-trivial hyperedges 
	(see Appendix \ref{app:2} for a proof).
	The fact that $\beta^d(d+1) \geq  d$ for odd values of $d$
	follows directly from Lemma \ref{lemma:lin}.
	\qed
\end{proof}


We have the following corollary
which gives an upper bound to the cardinality of an induced-bisecting families for arbitrary values of $n$.

\begin{corollary}
	Let $\mathcal{H}$ be any hypergraph on vertex set $V(\mathcal{H})=\{v_1,\ldots,v_n\}$ and let $d \in [n]$.
	Let $\mathcal{F}$ consist of $(d+1)$-sized subsets of $V(\mathcal{H})$ such that
	for every $B \in E(\mathcal{H})$, there exists an $A \in \mathcal{F}$ with 
	(i) $|B \cap A|\geq 2$, when $d$ is even;
	(ii) $2 \leq |B \cap A| \leq d$, when $d$ is odd.
	Then, we can construct an induced-bisecting family of order $d$ of cardinality
	$|\mathcal{F}|(d+1)$ for $\mathcal{H}$.
	\label{prop:intersect} 
\end{corollary}

\begin{proof}
	For any subset $A \in \mathcal{F}$, using the procedure used in the proof of Theorem \ref{thm:n-1},  we can obtain an induced-bisecting family $\mathcal{X}_{A}$ for 
	all the non-trivial subsets of $A$, where $|\mathcal{X}_{A}|=d+1$.
	When $d$ is even, $\mathcal{X}_{A}$ induced-bisects all the $2^{d+1}-(d+1)-1$ non-empty and non-singleton subsets of $A$;
	therefore, each $B \in E(\mathcal{H})$ with $|B \cap A|\geq 2$ is induced-bisected by $\mathcal{X}_{A}$.
	When $d$ is odd, $\mathcal{X}_{A}$ induced-bisects all but the empty set, the singleton sets, and $A$;
	so, each $B \in E(\mathcal{H})$ with $2 \leq |B \cap A|\leq d$ is induced-bisected by $\mathcal{X}_{A}$.
	Repeating the process for each $A \in \mathcal{F}$, we get an induced-bisecting family of cardinality
	$|\mathcal{F}|(d+1)$ for $\mathcal{H}$.
	\qed
\end{proof}

Theorem \ref{thm:n-1} provides evidence for the following property (which is described in Corollary \ref{cor:1}) of 
the odd subsets under any circular permutation of odd number of elements which
may be of independent interest.
For any circular permutation $\sigma$ of $[n]$, $a,b \in [n]$,
let $dist_{\sigma}(a,b)$ denote the clockwise distance between $a$ and $b$
with respect to  $\sigma$, which is one more than the number of elements residing between $a$ and $b$ in the permutation $\sigma$ in the clockwise direction.

\begin{corollary}\label{cor:1}
	Consider any circular permutation $\sigma$ of $[n]$, where $n$ is odd.
	For any odd $k$-sized subset $A \subseteq [n]$, let $(a_0,\ldots,a_{k-1})$ be the ordering of $A$ with respect to $\sigma$.
	Then, there exists an index $i \in \{0,\ldots,k-1\}$ such that $dist_{\sigma}(a_i,a_{i+\lfloor\frac{k}{2}\rfloor}) < \frac{n}{2}$
	and $dist_{\sigma}(a_{i+\lfloor\frac{k}{2}\rfloor+1},a_i) < \frac{n}{2}$, where summation in the subscript of $a$ is modulo $k$.
\end{corollary}
\begin{proof}
	Consider a circular clockwise arrangement of $n$ slots, namely $P_1,\ldots,P_n$ in that order. Put vertex $\sigma(i)$ in $P_i$.
	Now, following the procedure outlined in the proof of Theorem \ref{thm:n-1}, obtain a bicoloring that bisects $A$.
	Pick the uncolored vertex residing in slot $P_{\lceil \frac{n}{2}\rceil}$ with respect to the bicoloring $X$.
	Observe that this vertex satisfies the desired property.
	\qed
\end{proof}

\section{Upper bounds for $\beta^d(n)$ and proof of Theorem \ref{thm:main}}
\label{sec:randup}
From Proposition \ref{prop:low}, we know that 
$\beta^d(n) \geq \frac{2n(n-1)}{d^2}$.
In this section, we prove an upper bound of $\binom{\lceil\frac{2(n-1)}{d-1} \rceil}{2}+\lceil\frac{n-1}{d-1}\rceil (d+1)$ for  $\beta^d(n)$.

\subsection{A deterministic construction of induced-bisecting families}

\begin{lemma}\label{lemma:up}
$\beta^d(n) \leq \binom{\lceil\frac{2(n-1)}{d-1} \rceil}{2}+\lceil\frac{n-1}{d-1}\rceil (d+1)$.
\end{lemma}

Before proceeding to the proof of the above lemma, we give few definitions that
simplify the proof considerably.
Let $d$ be a positive even integer.
Let $\mathcal{S}(n,d)=\{P_1,\ldots, P_{\lceil\frac{2n}{d}\rceil}\}$ denote a partition of $[n]$, where each 
$P \in \mathcal{S}(n,d) \setminus \{P_{\lceil\frac{2n}{d}\rceil}\}$ is of cardinality exactly $\frac{d}{2}$, and 
$|P_{\lceil\frac{2n}{d}\rceil}| \leq \frac{d}{2} $.
Let $P_{\lceil\frac{2n}{d}\rceil}^1 = P_{\lceil\frac{2n}{d}\rceil} \cup Q_1$, 
$P_{\lceil\frac{2n}{d}\rceil}^2 = P_{\lceil\frac{2n}{d}\rceil} \cup Q_2$,
where $Q_i$ denotes a fixed $(\frac{d}{2} - |P_{\lceil\frac{2n}{d}\rceil}|)$-sized subset of $P_i$.
For an even $d$, we define $\mathcal{P}(n,d)$, $\mathcal{D}(n,d)$ and $\mathcal{B}(n,d)$ as follows.

\subsubsection*{Definition of $\mathcal{P}(n,d)$}

\begin{align}
\mathcal{P}(n,d)= 
\begin{cases}
\mathcal{S}(n,d), \text{ if $\frac{d}{2}$ divides $n$ } \\
\mathcal{S}(n,d) \setminus \{P_{\lceil\frac{2n}{d}\rceil}\} \cup \{P_{\lceil\frac{2n}{d}\rceil}^1,P_{\lceil\frac{2n}{d}\rceil}^2\}, \text{ otherwise.}
\end{cases}
\end{align}

\subsubsection*{Definition of $\mathcal{B}(n,d)$}

\paragraph*{$\frac{d}{2}$ divides $n$ :}
For each $i,j \in \left[\frac{2n}{d}\right]$, $i < j$,
let $B_{i,j}:P_i \cup P_j \rightarrow \{+1,-1\}$ denote a bicoloring, where 
\[B_{i,j}(x)=\begin{cases}
+1, \text{ if $x \in P_i$ } \\
-1, \text{ if $x \in P_j$.}
\end{cases}\]
Let $\mathcal{B}(n,d)=\{B_{i,j}|i,j \in \left[\frac{2n}{d}\right], i < j\}$ denote this set of bicolorings.

\paragraph*{$\frac{d}{2}$ does not divide $n$ :}
For each $i,j \in \left[\lceil\frac{2n}{d}\rceil-1\right]$, $i < j$,
let $B_{i,j}:P_i \cup P_j \rightarrow \{+1,-1\}$ denote a bicoloring, where 
\[B_{i,j}(x)=\begin{cases}
+1, \text{ if $x \in P_i$ } \\
-1, \text{ if $x \in P_j$.}
\end{cases}\]

Let $B_{1,\lceil\frac{2n}{d}\rceil}: P_1 \cup P_{\lceil\frac{2n}{d}\rceil}^2 \rightarrow \{-1,1\}$ and $B_{i,\lceil\frac{2n}{d}\rceil}: P_i \cup P_{\lceil\frac{2n}{d}\rceil}^1 \rightarrow \{-1,1\}$, for $2 \leq i \leq \lceil\frac{2n}{d}\rceil-1$.

\[B_{1,\lceil\frac{2n}{d}\rceil}(x)=\begin{cases}
+1, \text{ if $x \in P_1$ } \\
-1, \text{ if $x \in P_{\lceil\frac{2n}{d}\rceil}^2$}
\end{cases}\].

\[B_{i,\lceil\frac{2n}{d}\rceil}(x)=\begin{cases}
+1, \text{ if $x \in P_i$ } \\
-1, \text{ if $x \in P_{\lceil\frac{2n}{d}\rceil}^1$}
\end{cases} \text{, for $2 \leq i \leq \left\lceil\frac{2n}{d}\right\rceil-1$.}\]

Let $\mathcal{B}(n,d)=\{B_{i,j}|i,j \in \left[\lceil\frac{2n}{d}\rceil\right], i < j\}$ denote this set of bicolorings. 

\subsubsection*{Definition of $\mathcal{D}(n,d)$}
$\mathcal{D}(n,d)= \{D_k|D_k = P_{2k-1} \cup P_{2k}, k \in \left[\lceil\frac{n}{d}\rceil-1\right]\}\cup \{ D_{\lceil\frac{n}{d}\rceil}\} $,
where
\begin{align}
D_{\lceil\frac{n}{d}\rceil} = \begin{cases}
P_{\frac{2n}{d}-1} \cup P_{\frac{2n}{d}}, \text{ if $\frac{d}{2}$ divides $n$ }\\
P_1 \cup P_{\lceil\frac{2n}{d}\rceil}^2, \text{ if $\frac{d}{2}$ does not divide $n$ and $\lceil\frac{2n}{d}\rceil$ is odd }\\
P_{\lceil\frac{2n}{d}\rceil-1} \cup P_{\lceil\frac{2n}{d}\rceil}^2, \text{ if $\frac{d}{2}$ does not divide $n$ and $\lceil\frac{2n}{d}\rceil$ is even. }
\end{cases}
\end{align}

%
%
%
%
%
%
%
%

\begin{proof}
If $d=n-1$, the statement of the lemma follows directly from Theorem \ref{thm:n-1}. So, we assume that
$d < n-1$ in the rest of the proof.
We prove this lemma considering the exhaustive cases based on whether $d$ is even or odd, separately.

\subsubsection*{Case 1. $d$ is even}

Let $\mathcal{P}=\mathcal{P}(n,d)$, $\mathcal{B}=\mathcal{B}(n,d)$ and $\mathcal{D}=\mathcal{D}(n,d)$.

\begin{observation}
	For any $C \subseteq [n]$, $|C| \geq 2$, if $|C \cap P| \leq 1$, for all $P \in \mathcal{P}$,
	then $C$ is induced-bisected by at least one $B \in \mathcal{B}$.
	\label{claim:cm1}
\end{observation}

For any $C \subseteq [n]$, $|C| \geq 2$,
it follows from the premise that there exist $P_i,P_j \in \mathcal{P}$, $i < j$, such that
$|C \cap P_i|=|C \cap P_j|=1$. $C$ is induced-bisected by the bicoloring $B_{i,j}$, thus completing the proof of Observation \ref{claim:cm1}.

Let $\mathcal{C}$ denote the family of all the subsets of $[n]$ that are not induced-bisected by any $B \in \mathcal{B}$.
Rephrasing Observation \ref{claim:cm1}, for each $C \in \mathcal{C}$,
there exists a $P \in \mathcal{P}$ (and thus, a $D \in \mathcal{D}$) such that $|C \cap P| \geq 2$ (respectively, $|C \cap D| \geq 2$).
Let $\mathcal{D}'=\{D \cup \{j\}| j \in [n]\setminus D, D \in \mathcal{D}\}$.
Recall that $|D|=d$, where $d$ is an even integer less than $n-1$. So, each $D' \in \mathcal{D}'$ is a $(d+1)$-sized set. 
Using Corollary \ref{prop:intersect},
every $C \in \mathcal{C}$ can be induced-bisected using $|\mathcal{D}|(d+1)$ bicolorings.
Therefore, we have, $\beta^d(n) \leq |\mathcal{B}|+|\mathcal{D}|(d+1)=\binom{\lceil\frac{2n}{d} \rceil}{2}+\lceil \frac{n}{d}\rceil(d+1)$, when $d$ is even.

\subsubsection*{Case 2. $d$ is odd}

Let $\mathcal{P}=\mathcal{P}(n-1,d-1)$, $\mathcal{B}=\mathcal{B}(n-1,d-1)$ and $\mathcal{D}=\mathcal{D}(n-1,d-1)$.
Since $d-1$ is even, $\mathcal{P}$, $\mathcal{B}$ and $\mathcal{D}$ are well defined.
We extend the domain of each $B \in \mathcal{B}$ to $domain(B) \cup\{n\}$, and assign a +1 color to $n$ in each $B$.
Now, each $B \in \mathcal{B}$ colors exactly $d$ elements of $[n]$.

\begin{observation}
	For any $C \subseteq [n]$ with  $|C| \geq 2$, if $n \not\in C$ and $|C \cap P| \leq 1$ for all $P \in \mathcal{P}$,
	then $C$ is induced-bisected by at least one $B \in \mathcal{B}$.
	\label{claim:cm2}
\end{observation}

The proof of this observation is exactly the same as the proof of Observation \ref{claim:cm1}.

Let $\mathcal{C}$ denote the family of all the subsets of $[n]$ that are not induced-bisected by any $B \in \mathcal{B}$.
For any $D \subseteq [n]$, let $\max(D)$ denote the maximum integer in the set $D$.
Let $\mathcal{D}'=\{D \cup \{n\} \cup \{\max(D)+1\}| D \in \mathcal{D}\}$, where the addition is modulo $n-1$.

\begin{observation}\label{obs:lemma2}
Let $\mathcal{D}'=\{ D_1',D_2',...,D_{\lceil\frac{n-1}{d-1}\rceil}' \}$ be the family of subsets constructed as above.
Then, $|D_i' \cap D_{i+1}'|=2$, if $1 \leq i \leq \lceil\frac{n-1}{d-1}\rceil-1$, and 
$|D_{\lceil\frac{n-1}{d-1}\rceil}' \cap D_{1}'| \geq 2$.
\end{observation}

Recall that each $D \in \mathcal{D}$ is a $(d-1)$-sized subset of $[n-1]$, where $d$ is an odd integer less than $n-1$.
So, each $D' \in \mathcal{D}'$ is a $(d+1)$-sized set. 
From Observation \ref{claim:cm2}, it follows that 
for each $C \in \mathcal{C}$,
there exists at least one $D' \in \mathcal{D}'$ such that $|C \cap D'| \geq 2$.
Let $\mathcal{C}' \subseteq \mathcal{C}$ be the family of subsets of $[n]$ such that
for each $C' \in \mathcal{C}'$, there exists some $D' \in \mathcal{D}'$ such that $2 \leq |C \cap D'| \leq d$. 
Using Corollary \ref{prop:intersect},
we can obtain an induced-bisecting family for members of $\mathcal{C}'$ of cardinality $|\mathcal{D}|(d+1)$.
So, it follows that any $C \in \mathcal{C} \setminus \mathcal{C}'$ must contain one or more elements from
$\{ D_1',D_2',...,D_{\lceil\frac{n-1}{d-1}\rceil}' \}$ as its subsets.

For any $C \in \mathcal{C} \setminus \mathcal{C}'$, if $D_{i}' \subseteq C$,
then $D_{i+1}'\subseteq C$: otherwise, from Observation \ref{obs:lemma2}, $2 \leq |C \cap D_{i+1}'| \leq d$ and from definition of $\mathcal{C}'$, $C \in \mathcal{C}'$.
So, it follows that $\mathcal{C} \setminus \mathcal{C}'=\{[n]\}$, and $[n]$ is a trivial set when $d$ is odd.
Therefore, the cardinality of the induced-bisecting family for $[n]$ when $d$ is odd is at most
$|\mathcal{B}|+|\mathcal{D}|(d+1)=\binom{\lceil\frac{2(n-1)}{d-1} \rceil}{2}+\lceil\frac{n-1}{d-1}\rceil (d+1)$.

  \qed	
\end{proof}

\subsection{Proof of Theorem \ref{thm:main}}

\begin{statement}
	Let $2 \leq d \leq n$, where $d$ and $n$ are integers. Then,
$\frac{2n(n-1)}{d^2} \leq \beta^d(n) \leq \binom{\lceil\frac{2(n-1)}{d-1} \rceil}{2}+\lceil\frac{n-1}{d-1}\rceil (d+1)$.
Moreover, $\beta^d(n) \geq n-1$, when $d$ is odd.
\end{statement}

\begin{proof}
Theorem \ref{thm:main} follows from Proposition \ref{prop:low}, Lemma \ref{lemma:lin} and Lemma \ref{lemma:up}.
\qed
\end{proof}

\begin{remark}
	By removing some duplicate bicolorings, we can actually improve the upper bound for $\beta^d(n)$ from 
	$\binom{\lceil\frac{2(n-1)}{d-1} \rceil}{2}+\lceil\frac{n-1}{d-1}\rceil (d+1)$ to $\binom{\lceil\frac{2(n-1)}{d-1} \rceil}{2}+\lceil\frac{n-1}{d-1}\rceil d$.
\end{remark}

Theorem \ref{thm:main} asserts an upper bound of $O(n)$ on $\beta^d(n)$ when $d \in \Omega(\sqrt{n})$.
Let $k(G)$ denote the minimum cardinality of any hyperedge of the hypergraph $G$, i.e., $k(G)= \min_{e \in E(G)} |e|$.
For any hypergraph $G$, the upper bound for $\beta^d(G)$ can be improved to $O(n)$ even if 
$d \in o(\sqrt{n})$ provided $(d-1)k(G) > n-1$ in the following way.
Since $(d-1)k(G) > n-1$, 
every hyperedge is large enough so that the family $\mathcal{D}'$ constructed in all the cases of proof of Lemma \ref{lemma:up}
satisfies the conditions of the family requirements of Corollary \ref{cor:1}.
Therefore,
the set of bicolorings given by $\mathcal{B}=\mathcal{B}(n,d)$ (or $\mathcal{B}(n-1,d-1)$)
can be completely avoided.
Thus, we have the following theorem.

\begin{theorem}
For any hypergraph $G$, let $k(G) = \displaystyle\min_{e \in E(G)} |e|$.
If $(d-1)k(G) > n-1$, then
$\beta^d(G) \leq \lceil \frac{n-1}{d-1} \rceil(d+1)$.
\end{theorem}

\begin{remark}
The proof of Theorem \ref{thm:main} is algorithmic: it yields an induced bisecting family of cardinality at most 
$\binom{\lceil\frac{2(n-1)}{d-1} \rceil}{2}+\lceil\frac{n-1}{d-1}\rceil (d+1)$ with a running time of
$O(\frac{n^2}{d^2}+n)$. Observe that the running time of our algorithm is asymptotically equivalent to the cardinality of the family of bicolorings it outputs.
Therefore, the asymptotic running time of our algorithm is optimal whenever it outputs an asymptotically optimal solution.
Recall that Theorem \ref{thm:main} asserts tight bounds for $\beta^d(n)$ except for the even values of $d \in \Omega(n^{0.5+\epsilon})$, for any $\epsilon$, $0< \epsilon \leq 0.5$.
\end{remark}

We note that if $d=O(1)$, then Theorem \ref{thm:main} asserts that $\beta^d(n)=\Theta(n^2)$. However, the corresponding coefficients are not the same: the constant factor in the $\Omega(n^2)$ lower bound is $\frac{2}{d^2}$ whereas the constant factor in the upper bound is $\frac{2}{(d-1)^2}$. It would be interesting to determine the exact coefficient in this case.
Moreover, when $d$ is even and $d \in \Omega(n^{0.5+\epsilon})$, for any $\epsilon$, $0 < \epsilon \leq 0.5$, 
we have an upper bound of $O(n)$ on $\beta^d(n)$; the lower bound for this case is $o(n)$.
We  believe that $\beta^d(n)$ is more close to the upper bound and tightening of the bound for 
$\beta^d(n)$  in this case remains open.

{\small
	\bibliography{refer}}
\bibliographystyle{plain}

\appendix
\section{Proof of Theorem \ref{thm:n-1} when $d$ is odd} \label{app:2}
\begin{statement}
	$d \leq \beta^d(d+1) \leq d+1$, $d$ is an odd integer.
\end{statement}

\begin{proof}
As in the proof of Theorem \ref{thm:n-1}, the slots $P_1$ to $P_{\frac{d+1}{2}-1}$ are colored with +1, slots $P_{\frac{d+1}{2}+1}$ to $P_{d+1}$ are colored with -1, and only slot $P_{\frac{d+1}{2}}$ remains uncolored.
Note that along with the empty set and the singleton sets, the set $[d+1]$ becomes trivial under this restriction.
Each slot can contain exactly one vertex and each vertex takes the color of the slot it resides in.
As the initial configuration, let $v_i \in P_i$, for $1 \leq i \leq d+1$. This configuration gives the coloring $X_1$, where
(i) $X_1(+1)=\{v_1,\ldots, v_{\frac{d+1}{2}-1}\}$, (ii) $X_1(-1)=\{v_{\frac{d+1}{2}+1},\ldots, v_{d+1}\}$, and, (iii) the vertex $v_{\frac{d+1}{2}+1}$ remains uncolored.
We obtain the second coloring $X_2$ from $X_1$ by one clockwise rotation of the 
vertices in the circular arrangement.
Therefore, we have, $X_2(+1)=\{v_{d+1},v_1,\ldots, v_{\frac{d+1}{2}-2}\}$, $X_2(-1)=\{v_{\frac{d+1}{2}},\ldots, v_{d}\}$; the vertex $v_{\frac{d+1}{2}-1}$ remains uncolored.
Similarly, repeating the rotation $d$ times, we obtain the set $\mathcal{X}=\{X_1,\ldots,X_{d+1}\}$  of bicolorings.

The proof for $\mathcal{X}$ being an induced-bisecting family for any odd hyperedge $A_o \subsetneq [d+1]$ is exactly similar to that given in the proof of Theorem \ref{thm:n-1}.
So, we consider only the even hyperedges.
Let $c_i=|A \cap X_{i+1}(+1)|-|A \cap X_{i+1}(-1)|$, for $0 \leq i \leq d$.
All additions/subtractions in the subscript of $c$ are modulo $d+1$.
For the sake of contradiction, assume that $A$ is an even hyperedge not induced-bisected by $\mathcal{X}$.
If we can show that some $c_j$, $0 \leq j \leq d$, is zero, then we get the desired contradiction. 

\begin{observation}\label{obs:ltrep}
	$|c_i - c_{i+1}| \leq 2$,  for $0 \leq i \leq d$. 
\end{observation}
The proof of Observation \ref{obs:ltrep} follows from the construction.
Consider the sequence $(c_i,c_{i+1},\ldots, \allowbreak c_{i+d+1})$, where $c_i \leq c_j$, $j \in \{i+1,\ldots, i+d+1\}$, and the addition is modulo $d+1$. 
Since there is a full rotation  of the vertex set with respect to the slots, it follows that (i) $c_i \leq 0$, and (ii) there exists another index $j$ such that $c_j$ is positive.
From Observation \ref{obs:ltrep}, it follows that if none of the $c_j$, $j \in \{0, \ldots d\}$, is zero, there exists an index $p$ such that $c_p =-1$ and $c_{p+1}=1$.
Note that $c_p=-1$ asserts that $A \cap P_{\frac{d+1}{2}}$
is non-empty. However, under this configuration, $c_{p+1}$ can never become 1.
This yields the desired contradiction.
\qed
\end{proof}

\end{document}